\documentclass{amsart}
\usepackage{amsmath, amssymb, amscd, amsfonts}

\newtheorem{theorem}{Theorem}[section]

\newtheorem{corollary}[theorem]{Corollary}

\theoremstyle{definition}

\newtheorem{example}[theorem]{Example}
\newtheorem{remark}[theorem]{Remark}

\numberwithin{equation}{section}

\newcommand{\N}{\mathbb{N}}                        
\newcommand{\C}{\mathbb{C}}                        
\newcommand{\R}{\mathbb{R}}                        
\newcommand{\K}{\mathbb{K}}                        
\newcommand{\uk}{\underline{k}}                    
\newcommand{\vp}{\varphi}                          
\newcommand{\FF}{\mathcal{F}}                      
\newcommand{\NN}{\mathfrak{N}}                     
\newcommand{\OO}{\mathcal{O}}                      
\newcommand{\OXxi}{{\mathcal{O}}_{X,\xi}}          
\newcommand{\OYeta}{{\mathcal{O}}_{Y,\eta}}        
\newcommand{\OZzeta}{{\mathcal{O}}_{Z,\zeta}}      
\newcommand{\Xxi}{X_{\xi}}                         
\newcommand{\Yeta}{Y_{\eta}}                       
\newcommand{\vpxi}{\vp_{\xi}}                      
\newcommand{\fbd}{\mathrm{fbd}}                    
\newcommand{\antens}{\tilde{\otimes}}              
\newcommand{\tensR}{\tilde{\otimes}_R}             
\newcommand{\pp}{\mathfrak{p}}                     
\newcommand{\qq}{\mathfrak{q}}                     
\newcommand{\mm}{\mathfrak{m}}                     
\newcommand{\Spec}{\mathrm{Spec}\,}                
\newcommand{\supp}{\mathrm{supp}\,}                
\newcommand{\inexp}{\mathrm{exp}\,}                

\begin{document}

\title[A fast flatness testing criterion in characteristic zero]{A fast flatness testing criterion\\ in characteristic zero}

\author{Janusz Adamus}
\address{Department of Mathematics, The University of Western Ontario, London, Ontario, Canada N6A 5B7 -- and --
         Institute of Mathematics, Polish Academy of Sciences, ul. {\'S}niadeckich 8, 00-956 Warsaw, Poland}
\email{jadamus@uwo.ca}
\author{Hadi Seyedinejad}
\address{Department of Mathematics, The University of Western Ontario, London, Ontario, Canada N6A 5B7}
\curraddr{Department of Mathematical Sciences, University of Kashan, Kashan, Iran}
\email{seiiedine@kashanu.ac.ir}
\thanks{J. Adamus's research was partially supported by the Natural Sciences and Engineering Research Council of Canada}


\begin{abstract}
We prove a fast computable criterion that expresses non-flatness in terms of torsion:
Let $\vp:X\to Y$ be a morphism of real or complex analytic spaces and let $\xi$ be a point of $X$.
Let $\eta=\vp(\xi)\in Y$ and let $\sigma:Z\to Y$ be the blowing-up of $Y$ at $\eta$, with $\zeta\in\sigma^{-1}(\eta)$. Then $\vp$ is flat at $\xi$ if and only if the pull-back of $\vp$ by $\sigma$, $X\times_YZ\to Z$ has no torsion at $(\xi,\zeta)$; i.e., the local ring $\OO_{X\times_YZ,(\xi,\zeta)}$ is a torsion-free $\OZzeta$-module.
We also prove the corresponding result in the algebraic category over any field of characteristic zero.
\end{abstract}
\maketitle


\section{Introduction}
\label{sec:intro}

Flatness of a morphism $\vp:X\to Y$ of algebraic or analytic varieties is a fundamental property, which -- when present -- allows one to regard the fibres of $\vp$ as a family of varieties parametrized by a given variety $Y$. It is therefore interesting to know how to verify whether or not a given morphism is flat. However, determining flatness is, in general, a difficult task.
The purpose of the present paper is to give a criterion that expresses local flatness in terms of torsion-freeness and is easily computable using computer algebra.

Our flatness criterion (see Section~\ref{subsec:main-results}, below) asserts that non-flatness of $\vp$ at a point $\xi\in X$ can be detected by looking for torsion in the pullback of $\vp$ by the blowing-up of the target $Y$ at $\eta=\vp(\xi)$. In fact, our theorem is a more general criterion for flatness over $Y$ of a coherent sheaf of modules on $X$. The only restriction on $Y$ is that it be irreducible at the point $\eta$. Let $\K=\R$ or $\C$.

\begin{theorem}
\label{thm:1}
Let $\vp:X\to Y$ be a morphism of $\K$-analytic spaces, and let $\FF$ be a coherent sheaf of modules over $X$. Let $\xi$ be a point of $X$, and suppose that $Y$ is positive-dimensional and irreducible at $\eta=\vp(\xi)$.
Let $\sigma:Z\to Y$ be a local blowing-up of $\eta$ in $Y$, with $\zeta\in\sigma^{-1}(\eta)$, and let $z_{\mathrm{exc}}$ denote the local generator of the exceptional divisor at $\zeta$.
Then, $\FF_\xi$ is flat over $\OYeta$ if and only if $z_{\mathrm{exc}}$ is not a zerodivisor in $\FF_\xi\antens_{\OYeta}\OZzeta$.
\end{theorem}

(The analytic tensor product, denoted $\tensR$, is the coproduct in the category of local analytic $R$-algebras (see, e.g., \cite{GR} or \cite{ABM1}). For the basic facts and terminology on blowing-up, we refer the reader to \cite{Har}, \cite{H} or \cite{Hau}.) The simplicity of the above criterion from the computational point of view is best seen in Theorem~\ref{thm:alg-flat-regular}, below.
\medskip

The idea of expressing flatness in terms of zerodivisors dates back to Auslander's seminal paper \cite{Au}. Auslander showed that flatness of a finitely generated module over a regular local ring is equivalent to torsion-freeness of a sufficiently high tensor power of the module. In recent years, his theorem was extended to modules finite over an essentially finite-type morphism of schemes (or a holomorphic mapping of complex-analytic spaces): by Adamus, Bierstone and Milman \cite{ABM1} in the analytic and complex-algebraic categories (extending a special case done by Galligo and Kwieci{\'n}ski \cite{GK}), and by Avramov and Iyengar \cite{AI} in the category of schemes smooth over a field.

All of the above three generalizations follow the philosophy of Auslander's proof, and consequently all share the same limitations for practical application. First of all, they require the base ring to be regular (or even smooth, in \cite{AI}). This in itself is not yet the most restrictive assumption. As the authors show in \cite{AS}, it is not difficult to generalize to the singular case (at least in the complex-analytic and complex-algebraic categories). More importantly, in general, the above criteria detect non-flatness of a module by finding torsion only in its $n$-fold tensor power, where $n$ is the Krull dimension of the base ring. In fact, this is the case already for finite modules: Auslander \cite{Au} shows an example of a non-flat module $F$ finitely generated over a regular local ring $R$ of dimension $n$ such that $F$ as well as all its tensor powers up to $(n-1)$'st are torsion-free over $R$. To put this in a perspective of actual calculations, consider the module from Example~\ref{ex:non-flat}, below: There we have $R=\C[y_1,y_2,y_3]$ and a non-flat $R$-module $F$ finitely generated over $R[x_1,\dots,x_9]$. To verify the non-flatness of $F$ by means of \cite{GK}, \cite{ABM1} or \cite{AI}, one would need to perform primary decomposition (cf. \cite[Rem.\,1.4]{ABM1}) of an ideal in $3+3\cdot9=30$ variables. This is, of course, practically impossible.
Last but not least, in the analytic category, the existing flatness criteria only hold over the field of complex numbers.
\medskip

Let us consider now the geometric point of view. Let $\vp:X\to Y$ denote a morphism of algebraic or analytic varieties, with $Y$ smooth of dimension $n$, and $\vp(\xi)=\eta$. The general philosophy of \cite{ABM1} and \cite{AI} is that the non-flatness of $\vp$ at the point $\xi$ means that the fibre $\vp^{-1}(\eta)$ is in some sense bigger than the generic fibre of $\vp$. Passing to fibred powers of $\vp$ (which corresponds to taking tensor powers of the local ring $\OXxi$ of the source over $\OYeta$) amplifies the difference between the special and the generic fibre of the mapping to the extent that in the $n$-fold fibred power the special fibres themselves form an irreducible component of the source. (This component is responsible for the $\OYeta$-torsion in the $n$-fold tensor power of $\OXxi$.)

Here we take a different approach. The main idea behind our results is that the ``bigness'' of the fibre $\vp^{-1}(\eta)$ can be amplified much quicker. Namely, by taking fibred product with a morphism $\sigma:Z\to Y$ with generically finite fibres, whose fibre over $\eta$ is of codimension $1$ in $Z$. This technique was introduced by Hironaka in the proof of his local flattening theorem \cite[\S\,4, Thm.\,2]{H}. In fact, both our flateness criterion and Hironaka's flattening by shrinking of the special fibre follow from a more general Theorem~\ref{thm:main}, below.
\medskip

The main tools of this paper are Hironaka's diagram of initial exponents and the local flattener (recalled in Section~\ref{sec:toolbox}). Consequently, it is convenient to first prove our criterion in the language of local analytic algebras (over $\K=\R$ or $\C$), and later translate it into the analytic-geometric category. We then derive the corresponding result in the algebraic setting by standard faithfull flatness arguments.

\subsection{Main results}
\label{subsec:main-results}

Let $\K=\R$ or $\C$. Let $\vp:X\to Y$ be a morphism of $\K$-analytic spaces, and let $\FF$ be a coherent sheaf of modules over $X$.
Our main result asserts that non-flatness of $\FF$ at a point $\xi\in X$ can be detected by torsion in the pull-back of $\vp$ by a blowing-up of $Y$ at any local subspace $N$ at $\vp(\xi)$ with the property that $F_\xi$ is flat over $\OO_{N,\vp(\xi)}$. More precisely, we have the following.

\begin{theorem}
\label{thm:main}
Let $\vp:X\to Y$ be a morphism of $\K$-analytic spaces, and let $\FF$ be a coherent sheaf of modules over $X$. Let $\xi\in X$, and suppose that $\FF_\xi$ is not flat as an $\OYeta$-module, where $\eta=\vp(\xi)$. Let $N$ be a $\K$-analytic subspace of an open neighbourhood of $\eta$ in $Y$ such that $\eta\in N$ and $\OO_{N,\eta}\antens_{\OYeta}\FF_\xi$ is flat as an $\OO_{N,\eta}$-module. Let $\sigma:Z\to Y$ denote the local blowing-up of $N$ in $Y$, and let $\zeta\in\sigma^{-1}(\eta)$. Then the local generator $z_{\mathrm{exc}}$ of the exceptional divisor at $\zeta$ is a zerodivisor in $\OZzeta\antens_{\OYeta}\FF_\xi$.
\end{theorem}

\subsubsection*{Proof of Theorem~\ref{thm:1}}
The ``if'' direction of Theorem~\ref{thm:1} follows from Theorem~\ref{thm:main} applied with $N=\{\eta\}$, since every module is flat over a field. By irreducibility of $Y$ at $\eta$, the exceptional divisor (germ) $\{z_\mathrm{exc}=0\}$ contains no component of the strict transform (germ) of $Y$, hence $z_\mathrm{exc}$ is not a zerodivisor in $\OZzeta$. Therefore the ``only if'' direction is a consequence of the following simple remark and the fact that flatness is preserved by analytic base change (see, e.g., \cite[\S\,6, Prop.\,8]{H}).\qed

\begin{remark}
\label{rem:flat-tor-free}
Any flat $R$-module $M$ is torsion-free over $R$ (that is, a nonzero $r\in R$ is a zerodivisor in $M$ iff it is a zerodivisor in $R$). Indeed, this follows from the characterisation of flatness in terms of relations (see, e.g., \cite[Cor.\,6.5]{Eis}).
\end{remark}

Next, we would like to formulate a geometric variant of Theorem~\ref{thm:1} in the complex-analytic setting.
For that, we need to introduce the notion of vertical component.
Let, as above, $\vpxi:\Xxi\to\Yeta$ be a morphism of germs of $\K$-analytic spaces, and let $W_\xi$ denote an irreducible component of $\Xxi$ (isolated or embedded). Recall (\cite{ABM1}) that $W_\xi$ is called an \emph{algebraic} (resp. \emph{geometric}) \emph{vertical component} of $\vpxi$ (or over $\Yeta$) if $\vpxi$ maps $W_\xi$ to a proper analytic (resp. nowhere-dense) subgerm of $\Yeta$.
(More precisely, for a sufficiently small representative $W$ of $W_\xi$, and a corresponding representative $\vp$ of $\vpxi$, the germ $\vp(W)_\eta$ is a proper analytic (resp. nowhere-dense) subgerm of $\Yeta$.) A component of $\Xxi$ is isolated (resp. embedded) if its defining prime ideal in the local ring $\OXxi$ is an isolated (resp. embedded) prime.

\begin{corollary}
\label{cor:flat-cplx-map}
Let $\vp:X\to Y$ be a morphism of complex-analytic spaces, with $Y$ irreducible at a point $\eta$. Let $\sigma:Z\to Y$ denote the blowing-up of $\eta$ in $Y$. Let $\xi\in\vp^{-1}(\eta)$ and $\zeta\in\sigma^{-1}(\eta)$. Then, the map germ $\vpxi$ is flat if and only if its pull-back by $\sigma_\zeta$, $(X\times_YZ)_{(\xi,\zeta)}\to Z_\zeta$ has no algebraic vertical components. (Equivalently, the induced map germ $(X\times_YZ)_{(\xi,\zeta)}\to\Yeta$ has no algebraic vertical components.)
\end{corollary}

Corollary~\ref{cor:flat-cplx-map} follows from Theorem~\ref{thm:1} applied to $\FF=\OO_X$. Indeed, there is an isomorphism $\OXxi\antens_{\OYeta}\OZzeta\cong\OO_{X\times_YZ,(\xi,\zeta)}$, and every associated prime $\qq\in\OYeta$ of the $\OYeta$-module $\OO_{X\times_YZ,(\xi,\zeta)}$ is of the form $\qq=\pp\cap\OYeta$ for some associated prime $\pp$ of the ring $\OO_{X\times_YZ,(\xi,\zeta)}$. Therefore, $\OXxi\antens_{\OYeta}\OZzeta$ has a zerodivisor over $\OYeta$ if and only if some associated prime of $\OO_{X\times_YZ,(\xi,\zeta)}$ contracts to a non-zero ideal in $\OYeta$. In other words, the corresponding irreducible component of $X\times_YZ$ through $(\xi,\zeta)$ is mapped to a proper analytic subgerm of $Y_\eta$.

\begin{remark}
\label{rem:flat-map-case}
It is interesting to compare the above result (in the case of non-singular $Y$) with Theorem\,1.1 of \cite{ABM1}, where the non-flatness of $\vpxi$ is detected in the $n$-fold fibred power $\vp^{\{n\}}_{\xi^{\{n\}}}$. There, the characterisation is in terms of the \emph{geometric} vertical components, and it is actually an open problem (\cite[Question\,1.11]{ABM1}) whether it can be stated in terms of the algebraic vertical components as well. Clearly, every algebraic vertical component over an irreducible target is geometric vertical, but the converse is not true, in general (see, e.g., \cite{A1}).
\end{remark}
\medskip

Finally, let us state an algebraic analogue of (the exciting part of) Theorem~\ref{thm:1}. Here, as before, $\K=\R$ or $\C$.

\begin{theorem}
\label{thm:alg-flat-singular}
Let $R=\K[y_1,\dots,y_n]/I$, where $I$ is a proper ideal in $\K[y_1,\dots,y_n]$. Let $A=R[x_1,\dots,x_m]/Q$ be an $R$-algebra of finite type, and let $F$ be a finitely generated $A$-module.
Let $S=\K[z_1,\dots,z_n]$ and let $\kappa:\K[y_1,\dots,y_n]\to S$ be the morphism defined as
\[
\kappa(y_1)=z_1z_n, \dots, \ \kappa(y_{n-1})=z_{n-1}z_n, \ \kappa(y_n)=z_n\,.
\]
Let $I^*$ be the strict-transform ideal of $I$; i.e,
\[
I^*=\{g\in S: \ z_n^k\!\cdot\!g\in I\!\cdot\!S\ \mathrm{for\ some\ }k\in\N\}\,.
\]
Suppose that $I^*$ is a proper ideal in $S$. Then
$F_{(x,y)}$ is a flat $R_{(y)}$-module if $y_n$ is not a zerodivisor in $F_{(x,y)}\otimes_{R_{(y)}}S/I^*_{(z)}$.
\end{theorem}

\begin{remark}
\label{rem:weak-assumptions}
Notice the weakness of assumptions on $R$. In particular, the above criterion allows one to verify flatness of modules over a much larger class of rings than that in \cite[Thm.\,4.1]{AS}.
\end{remark}

In the case when the defining ideal $I$ of $R$ above is zero, Theorem~\ref{thm:alg-flat-singular} simplifies to the following criterion which holds over an arbitrary field of characteristic zero.

\begin{theorem}
\label{thm:alg-flat-regular}
Let $\uk$ be a field of characteristic zero and let $R=\uk[y_1,\dots,y_n]$.
Let $F$ be a module finitely generated over an $R$-algebra of finite type, say, $F\cong R[x]^q/M$, where $x=(x_1,\dots,x_m)$, $m\geq1$, and $M$ is a submodule of $R[x]^q$. Set $\widetilde{M}:=M(y_1y_n,\dots,y_{n-1}y_n,y_n,x)$, i.e., let $\widetilde{M}$ be the module obtained from $M$ by substituting $y_jy_n$ for $y_j$, $j=1,\dots,n-1$.
Then $F_{(x,y)}$ is a flat $R_{(y)}$-module if and only if \ $\widetilde{M}=\widetilde{M}:y_n$ (as $R[x]$-submodules of $R[x]^q$).
\end{theorem}

\subsection{Plan of the paper}

The rest of the paper is structured as follows: In Section~\ref{sec:toolbox}, we briefly recall Hironaka's combinatorial approach to flatness. We follow there the excellent exposition of \cite{BM}. Hironaka's apparatus is an essential component of the proof of our main result, Theorem~\ref{thm:main}. The latter is proved in Section~\ref{sec:flatness}.

In Section~\ref{sec:openness}, we prove a topological analogue of Theorem~\ref{thm:1} -- a criterion for (local) openness of a holomorphic mapping between complex-analytic spaces. Like our flatness criterion above, Theorem~\ref{thm:open} is superior to the known effective openness criteria (see \cite{A1} and \cite{ABM2}) from the computational point of view. On the other hand, Theorem~\ref{thm:open} requires an additional assumption that the source of the mapping be pure-dimensional. Example~\ref{rem:false-in-general} proves the necessity of this assumption.

Finally, in the last section, we give the proofs of our algebraic criteria, Theorems~\ref{thm:alg-flat-regular} and~\ref{thm:alg-flat-singular}. Roughly speaking, these follow from Theorem~\ref{thm:main} by faithfull flatness of completions of the rings of polynomials and the rings of convergent power series over the base ring. In Section~\ref{sec:alg}, we also give an example of an explicit calculation of non-flatness, showing Theorem~\ref{thm:alg-flat-regular} at work.

\section{Diagram of initial exponents and flatness}
\label{sec:toolbox}

Let $\K=\R$ or $\C$. Let $R=\K\{y\}/I$ be a local analytic $\K$-algebra with the maximal ideal $\mm$, where $y=(y_1,\dots,y_n)$ and $I\subset\K\{y\}$ is a proper ideal.
Let $x=(x_1,\dots,x_m)$ and define $R\{x\}:=\K\{y,x\}/I\!\cdot\!\K\{y,x\}$. Given $\beta=(\beta_1,\dots,\beta_m)\in\N^m$ and a positive integer $q$, we will denote by $x^\beta$ the monomial $x_1^{\beta_1}\dots x_m^{\beta_m}$, and by $x^{\beta,j}$ the $q$-tuple $(0,\dots,x^\beta,\dots,0)$ with $x^\beta$ in the $j$'th place. Then, a $q$-tuple $G=(G_1,\dots,G_q)\in R\{x\}^q$ can be written as $G=\sum_{\beta,j}g_{\beta,j}x^{\beta,j}$, for some $g_{\beta,j}\in R$, where the indices $(\beta,j)$ belong to $\N^m\times\{1,\dots,q\}$.

The mapping $R=\K\{y\}/I\to\K$, $g\mapsto g(0)$ of evaluation at zero, given by tensoring with $\otimes_R R/\mm$, induces the evaluation mapping
\[
R\{x\}^q\to\K\{x\}^q,\quad G=\sum_{\beta,j}g_{\beta,j}x^{\beta,j}\mapsto G(0)=\sum_{\beta,j}g_{\beta,j}(0)x^{\beta,j}\,.
\]
For a submodule $M$ of $R\{x\}^q$, we will denote by $M(0)$ the image of $M$ under the evaluation mapping.

Let $L$ be any positive linear form on $\R^m$, $L(\beta)=\sum_{i=1}^m\lambda_i\beta_i$ ($\lambda_i>0$). We define a total ordering of $\N^m\times\{1,\dots,q\}$ (denoted by $L$ again) by lexicographic ordering of the $(m+2)$-tuples $(L(\beta),j,\beta_1,\dots,\beta_m)$, where $\beta=(\beta_1,\dots,\beta_m)$ and $(\beta,j)\in\N^m\times\{1,\dots,q\}$.

For a $q$-tuple $G=\sum_{\beta,j}g_{\beta,j}x^{\beta,j}\in R\{x\}^q$, define the \emph{support} of $G$ as
\[
\supp(G)=\{(\beta,j):g_{\beta,j}\neq0\}\,.
\]
Similarly, for the evaluated $q$-tuple $G(0)$, we define
\[
\supp(G(0))=\{(\beta,j):g_{\beta,j}(0)\neq0\}\,.
\]
Of course, $\supp(G(0))\subset\supp(G)$.
The \emph{initial exponent} of $G(0)$, denoted $\inexp_L(G(0))$, is the minimum (with respect to $L$) over all $(\beta,j)$ in $\supp(G(0))$.

Given a submodule $M$ of $R\{x\}^q$, we will be interested in the \emph{diagram of initial exponents} of its evaluation at $0$ (with respect to $L$), defined as
\[
\NN_L(M(0)):=\{\inexp_L(G(0)):G\in M \mathrm{\ with\ }G(0)\neq0\} \ \subset \ \N^m\times\{1,\dots,q\}\,.
\]
Hironaka's flatness criterion (see \cite[Thm.\,7.9]{BM}) asserts that an $R$-module $R\{x\}^q/M$ is flat if and only if, for every $G\in M$, \ $\supp(G)\cap\NN_L(M(0))=\varnothing$ \,implies that \,$G=0$.
\smallskip

We will also need the following.

\begin{theorem}[{Hironaka, cf.\,\cite[Thm.\,7.12]{BM}}]
\label{thm:flattener}
Let $\vp:X\to Y$ be a morphism of $\K$-analytic spaces, with $\xi\in X$ and $\eta=\vp(\xi)\in Y$, and let $F$ be a finitely generated $\OXxi$-module. Then there exists a unique germ of a local $\K$-analytic subspace $\Sigma$ of $Y$ (i.e., a unique local analytic $\K$-algebra $\OO_{\Sigma,\eta}$ which is a quotient of $\OYeta$) such that:
\begin{itemize}
\item[(i)] $\OO_{\Sigma,\eta}\antens_{\OYeta}F$ is $\OO_{\Sigma,\eta}$-flat.
\item[(ii)] If $\iota:\Sigma\to Y$ is the inclusion, then for every morphism germ $\psi_\theta:T_\theta\to Y_\eta$ such that $\OO_{T,\theta}\antens_{\OYeta}F$ is $\OO_{T,\theta}$-flat there exists a unique morphism germ $\chi_\theta:T_\theta\to\Sigma_\eta$ such that $\psi_\theta=\iota_\eta\circ\chi_\theta$.
\end{itemize}
\end{theorem}

The unique germ $\Sigma_\eta$ of Theorem~\ref{thm:flattener} is called the \emph{flattener} of $F$, and its defining ideal is the \emph{flattener ideal} of $F$.

\section{Proof of Theorem~\ref{thm:main}}
\label{sec:flatness}

Let $\vp:X\to Y$, $\xi\in X$, $\eta=\vp(\xi)$, $\FF$, $N$, and $\sigma:Z\to Y$ be as in the statement of the theorem.
Let $n,k\geq1$ be such that $\OYeta$ (resp. $\OZzeta$) is isomorphic with a quotient of $\K\{y\}$ (resp. $\K\{z\}$), where $y=(y_1,\dots,y_n)$ and $z=(z_1,\dots,z_k)$.

Given a morphism $\vpxi:\Xxi\to\Yeta$ of germs of $\K$-analytic spaces, there is an open neighbourhood $V$ of $\xi$ in $X$ and a positive integer $m$ such that $V$ is a $\K$-analytic subspace of $\K^m$. Then $\vpxi$ factors as $(\pi_Y\circ\iota)_\xi$, where $\iota$ denotes the isomorphism $V\stackrel{\cong}{\rightarrow}\Gamma_{\vp|_V}\subset\K^m\times Y$ with the graph of $\vp$, and $\pi_Y:\K^m\times Y\to Y$ is the projection. It follows that $\OXxi$ can be identified with a quotient of $\OYeta\{x\}$, where $x=(x_1,\dots,x_m)$.

Consequently, $\FF_\xi$ being a finite $\OXxi$-module, we can write $\FF_\xi=\OYeta\{x\}^q/M$, where $p,q\geq1$ and $M$ is the image of a homomorphism of $\OYeta\{x\}$-modules $\OYeta\{x\}^p\stackrel{\Phi}{\rightarrow}\OYeta\{x\}^q$.
Then, by the right-exactness of tensor product and using the isomorphism $\OZzeta\antens_{\OYeta}\OYeta\{x\}\stackrel{\cong}{\rightarrow}\OZzeta\{x\}$, we can identify $\OZzeta\antens_{\OYeta}\FF_\xi$ with $\OZzeta\{x\}^q/M^Z$, where $M^Z$ denotes the image of $\Phi$ viewed as a homomorphism between $\OZzeta\{x\}^p$ and $\OZzeta\{x\}^q$.

Let $L$ be a positive linear form on $\R^m$. Let $M(0)$ (resp. $M^Z(0)$) denote the submodule of $\K\{x\}^q$ obtained by evaluating at zero the $y$-variables in $M$ (resp. the $z$-variables in $M^Z$). One readily sees that their diagrams of initial exponents are equal:
\[
\NN_L(M(0))=\NN_L(M^Z(0))\,.
\]
Indeed, this follows from the fact that $\sigma^*_\zeta:\OYeta\to\OZzeta$ is a local homomorphism.
Set
\[
H:=\{F\in M:\supp(F)\cap\NN_L(M(0))=\varnothing\}
\]
and
\[
H^Z:=\{G\in M^Z:\supp(G)\cap\NN_L(M^Z(0))=\varnothing\}\,.
\]
Each $F\in H$ (resp. $G\in H^Z$) can be written as $F=\sum_{(\beta,j)\notin\NN_L(M(0))}f_{\beta,j}x^{\beta,j}$ with $f_{\beta,j}\in\OYeta$ (resp. $G=\sum_{(\beta,j)\notin\NN_L(M^Z(0))}g_{\beta,j}x^{\beta,j}$ with $g_{\beta,j}\in\OZzeta$).
Define $P$ (resp. $P^Z$) as the ideal in $\OYeta$ (resp. in $\OZzeta$) generated by all the coefficients $f_{\beta,j}$ of all $F\in H$ (resp. all the $g_{\beta,j}$ over all $G\in H^Z$). It follows that $P^Z=P\cdot\OZzeta$ (i.e., $P^Z$ is the extension of $P$ by $\sigma^*_\zeta:\OYeta\to\OZzeta$).

By the proof of \cite[Thm.\,7.12]{BM}, $P$ is precisely the flattener ideal of $\FF_\xi$. Let $\Sigma$ denote a local $\K$-analytic subspace of $Y$ at $\eta$ such that $\OO_{\Sigma,\eta}=\OYeta/P$. By definition of the flattener, we have $N_\eta\subset\Sigma_\eta$ or, equivalently, $Q\supset P$, where $Q$ is the ideal in $\OYeta$ for which $\OO_{N,\eta}=\OYeta/Q$.

Now, as $\sigma:Z\to Y$ is the (local) blowing-up with centre $N$, it follows that the ideal $Q\cdot\OZzeta$ is generated by a single element, namely $z_\mathrm{exc}$. We thus have
\begin{equation}
\label{eq:exc}
(z_\mathrm{exc})\cdot\OZzeta\supset P\cdot\OZzeta=P^Z\,.
\end{equation}
By restricting to small neighbourhoods of $\zeta$ and $\eta$ (in $\K^k$ and $\K^n$, respectively), we can regard $\sigma:Z\to Y$ as a restriction of the blowing-up of $N$ in $\K^n$ to the strict transform $Z$ of $Y$. Then, since the exceptional divisor of $\sigma$ is locally a coordinate hypersurface, we can assume that $z_\mathrm{exc}$ is one of the coordinates at $\zeta$ (of the (local) ambient space $\K^k$ of $Z$ at $\zeta$), say, $z_\mathrm{exc}=z_k$. Note that no power of $z_\mathrm{exc}$ is zero in $\OZzeta$, and so we can identify $z_k$ with its class in $\OZzeta$.
\medskip

The ideal $P^Z$ is non-zero, since $\OZzeta\antens_{\OYeta}\FF_\xi$ is not $\OZzeta$-flat, by Theorem~\ref{thm:flattener}(ii) (see \cite[\S\,2]{AS} for a detailed argument). Hence, we can choose a non-zero $G^*\in H^Z$. Write $G^*=\sum g_{\beta,j}x^{\beta,j}$, with $g_{\beta,j}\in\OZzeta$. By \eqref{eq:exc} and definition of $P^Z$, each $g_{\beta,j}$ is divisible by $z_k$. Set
\[
d:=\max\{l\in\N:g_{\beta,j}\in(z_k)^l\!\cdot\!\OZzeta\mathrm{\ for\ all\ }(\beta,j)\in\supp(G^*)\}\,,
\]
and choose $(\beta^*,j^*)\in\supp(G^*)$ such that $g_{\beta^*,j^*}\in(z_k)^d\setminus(z_k)^{d+1}$.
We can now define
\begin{equation}
\label{eq:tilde}
\tilde{G}:=z_k^{-d}\cdot G^*\in\OZzeta\{x\}^q\,.
\end{equation}
Note that $\tilde{G}\notin M^Z$. Indeed, we have $\supp(\tilde{G})=\supp(G)$, hence $\supp(\tilde{G})\cap\NN_L(M^Z(0))=\varnothing$.
Thus, if $\tilde{G}=\sum \tilde{g}_{\beta,j}x^{\beta,j}$ belonged to $M^Z$ then all its coefficients $\tilde{g}_{\beta,j}$ would be in $P^Z$. But $\tilde{g}_{\beta^*,j^*}$ is not in $(z_k)\cdot\OZzeta$, by construction, and hence not in $P^Z$, by \eqref{eq:exc}.

On the other hand, $G^*$ was chosen from $M^Z$, and so $z_\mathrm{exc}^d$ (hence also $z_\mathrm{exc}$) is a zerodivisor in $\OZzeta\{x\}^q/M^Z$, by \eqref{eq:tilde}. This completes the proof of the theorem.
\qed

\section{Openness criterion}
\label{sec:openness}

Let $\vp:X\to Y$ be a morphism of complex-analytic spaces, with $\xi\in X$ and $\eta=\vp(\xi)\in Y$. Assume that $X$ is of pure dimension, $Y$ is positive-dimensional and locally irreducible. Let $\sigma:Z\to Y$ denote the (local) blowing-up of $\eta$ in $Y$, and let $\zeta\in\sigma^{-1}(\eta)$. The following is a topological analogue of Theorem~\ref{thm:1}.

\begin{theorem}
\label{thm:open}
The map $\vp$ is open at $\xi\in X$ if and only if its pull-back $X\times_YZ\to Z$ has no isolated algebraic vertical component at $(\xi,\zeta)$. (Equivalently, the induced map $X\times_YZ\to Y$ has no isolated algebraic vertical component at $(\xi,\zeta)$.)
\end{theorem}

Here, by a map $\vp$ \emph{open at a point $\xi$} we mean a map whose restriction to a certain open neighbourhood of $\xi$ is an open map. In other words, $\vpxi$ is a germ of an open map.

\begin{remark}
\label{rem:universally-open}
Note that in the category of continuous maps between topological spaces (hence also in the category of holomorphic maps between complex-analytic spaces), openness implies universal openness. That is, if $\vpxi:\Xxi\to\Yeta$ is an open germ and $\psi_\zeta:Z_\zeta\to\Yeta$ is an arbitrary map-germ, then the pull-back of $\vpxi$ by $\psi_\zeta$, $(X\times_YZ)_{(\xi,\zeta)}\to Z_\zeta$ is again open.
\end{remark}

\subsubsection*{Proof of Theorem~\ref{thm:open}}
If $\vpxi$ is open, then its pull-back by $\sigma_\zeta$, $(X\times_YZ)_{(\xi,\zeta)}\to Z_\zeta$ is also open, by Remark~\ref{rem:universally-open}. Therefore, $(X\times_YZ)_{(\xi,\zeta)}$ has no isolated algebraic vertical component over $Z_\zeta$, hence also over $Y_\eta$ (because $\sigma_\zeta$ is dominant).

Conversely, suppose that $\vpxi$ is not open. Then, by the Remmert Open Mapping Theorem (see \cite[{\S}V.6,\,Thm.\,2]{Loj}), we have $\fbd_\xi\vp>\dim{X}-\dim_\eta Y$, or
\begin{equation}
\label{eq:fbd}
\dim{X}\leq\dim_\eta Y-1+\fbd_\xi\vp\,,
\end{equation}
where $\fbd_x\vp$ denotes the fibre dimension of $\vp$ at a point $x$, $\dim_x\vp^{-1}(\vp(x))$.

The problem being local, we can assume that $Y$ is a subspace of $\C^n$ and $\sigma$ is the restriction of the blowing-up of $\eta=0$ in $\C^n$ to the strict transform $Z$ of $Y$. Further, assume that the local generator of the exceptional divisor is the coordinate $y_n$; i.e., in local coordinates $\sigma$ is defined as $y_n\mapsto y_n$, $y_j\mapsto y_jy_n$ for $j=1,\dots,n-1$.

Since $\sigma$ is a biholomorphism outside $\sigma^{-1}(\eta)$, we can write $X\times_YZ=T\cup T'$, where $T'=\sigma^{-1}(\eta)\times\vp^{-1}(\eta)$ and $T$ is biholomorphic with $\vp^{-1}(Y\setminus\{y_n=0\})$.
One readily sees that $\dim{T}\leq\dim{X}$ and $\dim_\zeta\sigma^{-1}(\eta)=\dim_\eta Y-1$. Therefore, by \eqref{eq:fbd},
\[
\dim{T}\leq\dim_\eta Y-1+\fbd_\xi\vp=\dim_{(\xi,\zeta)}T'\,.
\]
It follows that $\dim_{(\xi,\zeta)}T'=\dim_{(\xi,\zeta)}(X\times_YZ)$, and hence $T'$ must contain an isolated irreducible component of $X\times_YZ$ through $(\xi,\zeta)$. By definition of $T'$, such a component is mapped into $\sigma^{-1}(\eta)$ in $Z$, and so it is algebraic vertical (over $Z_\zeta$, as well as over $\Yeta$).
\qed
\medskip

In the algebraic setting, Theorem~\ref{thm:open} has the following analogue over an arbitrary field $\uk$ (cf. \cite[Thm.\,1.1]{ABM2}):

\begin{theorem}
\label{thm:open-alg}
Let $Y$ be a scheme of finite type over a field $\uk$, and let $\vp:X\to Y$ be a morphism which is locally of finite type.
Assume that $Y$ is normal and positive-dimensional and $X$ is of pure dimension.
Let $\mm$ be a closed point of $Y$, and let $\sigma:Z\to Y$ denote the blowing-up of $Y$ at $\mm$. Then $\vp$ is open at a point $\pp\in\vp^{-1}(\mm)$ if and only if the pullback of $\vp$ by $\sigma$, $X\times_YZ\to Z$ has no vertical irreducible components.
\end{theorem}

Openness of $\vp$ at $\pp$ means, as above, openness in some neighbourhood of $\pp$. A \emph{vertical} irreducible component is (by analogy with the local analytic case) an irreducible component of the source whose image is nowhere-dense in the target.

The proof of Theorem~\ref{thm:open-alg} is virtually identical with the one above, because: (a) openness of a map with a normal target is equivalent to universal openness (see \cite[Cor.\,14.4.3]{Gro}), and (b) a map $\vp:X\to Y$ with $X$ pure-dimensional and $Y$ normal is open if and only if $\vp$ is dominant and the fibres of $\vp$ are equidimensional and of constant dimension (see \cite[Cor.\,14.4.6]{Gro}).
\medskip

\begin{remark}
\label{rem:false-in-general}
Interestingly, Theorems~\ref{thm:open} and~\ref{thm:open-alg} are false, in general, without the pure-dimensionality assumption on $X$. This can be seen in the following example.
\end{remark}

\begin{example}
\label{ex:1}
Let $X=X_1\cup X_2$ be a subset of $\C^9$ (with coordinates $(t,x)=(t_1,t_2,t_3,x_1,\dots,x_6)$), where
\begin{align}
\notag
X_1&=\{(t,x):\, t_1x_1+t_2x_2+t_3x_3=t_2x_1+t_1x_2=x_4=x_5=x_6=0\}\,,\\
\notag
X_2&=\{(t,x):\, t_1=t_2=t_3=0\}\,.
\end{align}
Clearly, $X_2$ is irreducible, of dimension $6$. We claim that $X_1$ is of pure dimension $4$.
To see this, set $A=\{(t,x)\in X_1: \det \begin{bmatrix} t_1 & t_2 \\ t_2 & t_1\end{bmatrix}=0\}$.
In $X_1\setminus A$, one can solve the first two defining equations of $X_1$ for $x_1$ and $x_2$, hence $X_1\setminus A$ is a $4$-dimensional manifold. On the other hand, it is not difficult to see that $\dim{A}=3$. Since $X_1$ is defined by $5$ equations in $\C^9$, it follows that $\dim_\xi X_1\geq4$ for every $\xi\in X_1$. Therefore $A$ is nowhere-dense in $X_1$ and $X_1=\overline{X_1\setminus A}$ is of pure dimension $4$.

Define $\vp:X\to Y=\C^3$ as
\[
(t,x)\,\mapsto\,(t_1+x_4,t_2+x_5,t_3+x_6)\,.
\]
We claim that $\vp$ is not open at $0$. For this, it suffices to show that $\vp|_{X_1}$ is not open in any neighbourhood of $0$.
Consider the set $W=\{(t,x)\in X_1:\, t_3=0,\, t_1=t_2\neq0\}$. Then $W\subset X_1\setminus X_2$, and for every $\xi\in W$, we have $\fbd_\xi\vp=2$. On the other hand, the generic fibre dimension of $\vp|_{X_1}$ is $1$, as is easy to see. Therefore $\vp|_{X_1}$ is not open at any such $\xi$, by the Remmert Open Mapping Theorem.
But $W$ is adherent to $0\in\C^9$, which proves our claim.
\smallskip

Finally, let $\sigma:\C^3\to Y$ be given as $\sigma(z_1,z_2,z_3)=(z_1z_3,z_2z_3,z_3)$.
We shall show that $\vp':X\times_YZ\to Z$, the pullback of $\vp$ by $\sigma$ has no isolated vertical components through $(0,0)$.
As in the proof of Theorem~\ref{thm:open}, write $X\times_YZ=T\cup T'$, where $T'=\vp'^{-1}(0)$ and $T$ is biholomorphic with $X\setminus\vp^{-1}(\{y_3=0\})$. Since $\vp_0$ has no isolated algebraic vertical components itself, it follows that the image of (an arbitrarily small neighbourhood of a point $\xi$ near $0$ in) $T$ under $\vp'$ contains an open subset of $Y$. Therefore, if $\vp'$ has an isolated algebraic vertical component $\Sigma$ through $(0,0)$, then $\Sigma\subset T'$. But $T'$ is a fibre of an open map $\psi'$ defined as the pull-back by $\sigma$ of $\psi:=\vp|_{X_2}$. By Theorem~\ref{thm:open}, $\psi'$ has no isolated algebraic vertical components, which proves that there is no such $\Sigma$.
\end{example}

\section{Algebraic case}
\label{sec:alg}

Our flatness criterion in the algebraic setting can be reduced to the analytic case, settled above, by means of the following simple but fundamental observation.

\begin{remark}
\label{rem:ff}
Let $\K=\R$ or $\C$. Suppose that $M$ is a module over the local ring $R=\K[x]_{(x)}$, where $x=(x_1,\dots,x_n)$. Then
\begin{itemize}
\item[(i)] $M$ is a flat $R$-module if and only if $M\cdot\K\{x\}$ is a flat $\K\{x\}$-module.
\item[(ii)] Given non-zero $r\in R$, $r$ is a zerodivisor in $M$ if and only if $r$ is a zerodivisor in $M\cdot\K\{x\}$.
\end{itemize}
Indeed, the modules $M$ and $M\cdot\K\{x\}$ share the same completion $\widehat{M}$ (with respect to the $(x)$-adic topology in $R$ and $\K\{x\}$, respectively) over $\K[[x]]=\widehat{R}=\widehat{\K\{x\}}$.
By faitfull flatness of $\widehat{R}$ over $R$ (\cite[Ch.\,III, \S\,3, Prop.\,6]{Bou} and \cite[Ch.\,III, \S\,5, Prop.\,9]{Bou}), $M$ is $R$-flat if and only if $\widehat{M}$ is $\K[[x]]$-flat. By faitfull flatness of $\widehat{\K\{x\}}$ over $\K\{x\}$, in turn, $\widehat{M}$ is $\K[[x]]$-flat if and only if $M\cdot\K\{x\}$ is $\K\{x\}$-flat.

Claim (ii) follows from faithfull flatness of completion applied to the sequence $0\to M\stackrel{\cdot r}{\rightarrow}M$.
\end{remark}
\medskip

\subsubsection*{Proof of Theorem~\ref{thm:alg-flat-regular}}
Let $\uk$, $R=\uk[y_1,\dots,y_n]$, and $F\cong R[x]^q/M$ be as in the statement of the theorem.
Let $S=\K[z_1,\dots,z_n]$ and let $\kappa:R\to S$ be the morphism
\[
\kappa(y_1)=z_1z_n, \dots, \ \kappa(y_{n-1})=z_{n-1}z_n, \ \kappa(y_n)=z_n\,.
\]
Then, the condition \ $\widetilde{M}=\widetilde{M}:y_n$ is equivalent to saying that $y_n$ is not a zerodivisor in $F_{(x,y)}\otimes_{R_{(y)}}S_{(z)}$.

Suppose first that $F_{(x,y)}$ is a flat $R_{(y)}$-module. Since flatness is preserved by base change, it follows that $F_{(x,y)}\otimes_{R_{(y)}}S_{(z)}$ is flat and hence torsion-free over $S_{(z)}$ (by Remark~\ref{rem:flat-tor-free}). Hence also $F_{(x,y)}\otimes_{R_{(y)}}S_{(z)}$ is torsion-free over $R_{(y)}$, because $R_{(y)}$ embeds into $S_{(z)}$. In particular, $y_n$ is not a zerodivisor in $F_{(x,y)}\otimes_{R_{(y)}}S_{(z)}$.
\smallskip

Conversely, suppose that $F_{(x,y)}$ is not flat over $R_{(y)}$.
We will proceed in three steps, depending on $\uk$. First, suppose that $\uk=\C$.
Then $y_n$ is a zerodivisor on $F_{(x,y)}\otimes_{R_{(y)}}S_{(z)}$, by Theorem~\ref{thm:1} and Remark~\ref{rem:ff}.
\smallskip

Next, suppose that $\uk$ is algebraically closed. Then our result follows from the above case, by the Tarski-Lefschetz Principle (see, e.g., \cite{Se}), as flatness can be expressed in terms of a finite number of relations (\cite[Cor.\,6.5]{Eis}).
\smallskip

Finally, let $\uk$ be an arbitrary field of characteristic zero, and let $\K$ denote an algebraic closure of $\uk$. Set $R':=R\otimes_{\uk}\K$, $S':=S\otimes_{\uk}\K$, and $F':=F\otimes_{\uk}\K$. It is not difficult to verify that $R'$ is a faithfully flat $R$-module (see, e.g., \cite{ABM2}). Therefore, $F_{(x,y)}$ is not $R_{(y)}$-flat if and only if $F'_{(x,y)}$ is not $R'_{(y)}$-flat. By the previous part of the proof, the latter implies that $y_n$ is a zerodivisor in $F'_{(x,y)}\otimes_{R'_{(y)}}S'_{(z)}$. But $F'_{(x,y)}\otimes_{R'_{(y)}}S'_{(z)}\cong(F_{(x,y)}\otimes_{R_{(y)}}S_{(z)})\otimes_{R_{(y)}}R'_{(y)}$, so $y_n$ is also a zerodivisor in $F_{(x,y)}\otimes_{R_{(y)}}S_{(z)}$, which completes the proof.
\qed
\medskip

\subsubsection*{Proof of Theorem~\ref{thm:alg-flat-singular}}
Let $R=\K[y_1,\dots,y_n]/I$, where $I$ is a proper ideal in $\K[y_1,\dots,y_n]$. Let $A=R[x_1,\dots,x_m]/Q$ be an $R$-algebra of finite type, and let $F$ be a finitely generated $A$-module. Suppose that $F_{(x,y)}$ is not flat over $R_{(y)}$.

Let $\vp:X\to Y$ be the $\K$-analytic mapping of $\K$-analytic spaces associated to the morphism $\Spec{A}\to\Spec{R}$, and let $\widetilde{F}$ denote the finite $\OO_{X,0}$-module $F_{(x,y)}\cdot\OO_{X,0}$. The problem being local, we can assume that $Y$ is a subspace of $\K^n$. Let further $\sigma:\K^n\to\K^n$ be the mapping sending $(z_1,\dots,z_{n-1},z_n)$ to $(z_1z_n,\dots,z_{n-1}z_n,z_n)$, so that the pull-back homomorphism $\sigma^*_0:\K\{y\}\to\K\{z\}$ is given by the same formulas as $\kappa$ in the statement of the theorem. By assumption on the ideal $I^*$, the strict transform $Z$ of $Y$ (under $\sigma$) passes through the origin (in $\K^n$ with the $z$-variables), and so $(\sigma|_Z)_0:Z_0\to Y_0$ is a germ of the blowing-up of $\{0\}$ in $Y$, as in Theorem~\ref{thm:main}.

Now, $\widetilde{F}$ is a non-flat $\OO_{Y,0}$-module, by Remark~\ref{rem:ff}. Hence, by Theorem~\ref{thm:main}, $y_n$ is a zerodivisor in $\widetilde{F}\antens_{\OO_{Y,0}}\OO_{Z,0}$. Thus, by Remark~\ref{rem:ff} again, $y_n$ is a zerodivisor in $F_{(x,y)}\otimes_{R_{(y)}}S/I^*_{(z)}$, as required.
\qed
\medskip

\begin{example}
\label{ex:non-flat}
Consider the polynomial mapping $\vp:X\to Y=\C^3$ from Example~\ref{ex:1}.
That is, let $X=X_1\cup X_2$ be a subset of $\C^9$ (with coordinates $(t,x)=(t_1,t_2,t_3,x_1,\dots,x_6)$), where
\begin{align*}
X_1&=\{(t,x):\, t_1x_1+t_2x_2+t_3x_3=t_2x_1+t_1x_2=x_4=x_5=x_6=0\}\,,\\
X_2&=\{(t,x):\, t_1=t_2=t_3=0\}\,,
\end{align*}
and let
\[
\vp(t,x)\,=\,(t_1+x_4,t_2+x_5,t_3+x_6)\,.
\]
By Example~\ref{ex:1}, $\vp$ is not open at $0$. Since flatness implies openness, by a theorem of Douady (\cite{Dou}), it follows that $\vp$ is not flat at $0$. This can be verified directly, by means of Theorem~\ref{thm:alg-flat-regular}, as follows:

$X$ can be embedded into $\C^9\times Y$ via the graph of $\vp$. Hence, the coordinate ring $A[X]$ of $X$ can be identified with $\C[y,t,x]/(I_1+I_2)$, where
\begin{align*}
I_1&=(y_1-t_1-x_4,\,y_2-t_2-x_5,\,y_3-t_3-x_6)\quad\mathrm{and}\\
I_2&=(t_1x_1+t_2x_2+t_3x_3,\,t_2x_1+t_1x_2,\,x_4,\,x_5,\,x_6)\cap(t_1,\,t_2,\,t_3)\,.
\end{align*}
Set $F=A[X]$ and $R=\C[y]$. We want to prove that $F_{(t,x,y)}$ is not flat over $R_{(y)}$.

Let $\widetilde{I}_1$ (resp. $\widetilde{I}_2$) denote the ideal obtained from $I_1$ (resp. $I_2$) by substituting $y_1y_3$ for $y_1$, and $y_2y_3$ for $y_2$; i.e.,
\begin{align*}
\widetilde{I}_1&=(y_1y_3-t_1-x_4,\,y_2y_3-t_2-x_5,\,y_3-t_3-x_6)\quad\mathrm{and}\\
\widetilde{I}_2&=(t_1x_1+t_2x_2+t_3x_3,\,t_2x_1+t_1x_2,\,x_4,\,x_5,\,x_6)\cap(t_1,\,t_2,\,t_3)\,.
\end{align*}
By Theorem~\ref{thm:alg-flat-regular}, the non-flatness of $F_{(t,x,y)}$ over $R_{(y)}$ can be detected by showing that $(\widetilde{I}_1+\widetilde{I}_2):y_n\varsupsetneq \widetilde{I}_1+\widetilde{I}_2$.
We have verified, with help of a computer algebra system Singular (see, e.g., \cite{GP}), that $(\widetilde{I}_1+\widetilde{I}_2):y_n$ contains an element $x_6y_2-x_5$, which does not belong to $\widetilde{I}_1+\widetilde{I}_2$. This is easy to see.
Indeed, on the one hand,
\begin{multline*}
y_3(x_6y_2-x_5)=x_6y_2y_3-x_5y_3\equiv_{\widetilde{I}_1}x_6(t_2+x_5)-x_5y_3=t_2x_6+x_5x_6-x_5y_3\\
\equiv_{\widetilde{I}_2}0+x_5x_6-x_5y_3\equiv_{\widetilde{I}_1}x_5x_6-x_5(t_3+x_6)=-t_3x_5\equiv_{\widetilde{I}_2}0\,.
\end{multline*}
On the other hand, suppose that $x_6y_2-x_5\in\widetilde{I}_1+\widetilde{I}_2$. Then, after evaluating at zero the variables
$y_1, y_2, y_3, t_1, t_3, x_1, x_2, x_3, x_4$, and $x_6$, we would get
\[
-x_5 \in (t_2+x_5,t_2x_5)\!\cdot\!\C[t_2,x_5]\,,
\]
which is false.
\end{example}

\section*{Acknowledgements}

We would like to thank the anonymous referee for many valuable remarks that allowed us to improve the exposition of our results. In particular, we are indebted to the referee for suggesting the present form of Theorem~\ref{thm:main}.


\end{document}